\documentclass[12pt,article]{amsart}

\usepackage[margin=1in]{geometry}
\usepackage{amsmath, amssymb,amsthm,url,mathrsfs,graphicx,amscd,amsfonts}
\usepackage[mathscr]{eucal}
\usepackage{dsfont}
\usepackage{mathtools}
\usepackage{hyperref}

\usepackage[utf8]{inputenc}
\usepackage{graphicx,todonotes,hyperref}
\usepackage{epstopdf}
\usepackage{inputenc}
\usepackage{enumitem}
\usepackage{amsmath}
\usepackage{graphicx,todonotes,hyperref}
\usepackage{xcolor}
\usepackage{amssymb}

\newtheorem*{acknowledgements*}{Acknowledgements}

\allowdisplaybreaks
\newtheorem{theorem}{Theorem}[section]
\newtheorem{lemma}[theorem]{Lemma}

\newtheorem*{definition}{Definition}
\newtheorem*{theorem*}{Theorem}
\theoremstyle{remark}
\newtheorem{remark}{Remark}
\newtheorem*{remark*}{Remark}

\numberwithin{equation}{section}

\renewcommand{\phi}{\varphi}

\newcommand{\R}{\mathbb{R}}

\newcommand{\E}{\mathbb{E}}
\newcommand{\Z}{\mathbb{Z}}
\newcommand{\N}{\mathbb{N}}
\newcommand{\Var}{\text{Var}}
\newcommand{\boldb}{\boldsymbol{b}}
\newcommand{\boldc}{\boldsymbol{c}}
\newcommand{\boldk}{\boldsymbol{k}}
\newcommand{\boldn}{\boldsymbol{n}}

\newcommand{\boldv}{\boldsymbol{v}}
\newcommand{\boldq}{\boldsymbol{q}}
\newcommand{\boldz}{\boldsymbol{z}}
\newcommand{\boldx}{\boldsymbol{x}}

\newcommand{\boldzn}{\boldsymbol{z}_n}

\newcommand{\bolda}{\boldsymbol{a}}
\newcommand{\boldA}{\boldsymbol{A}}
\newcommand{\boldxn}{\boldsymbol{x}_n}
\newcommand{\boldCN}{\boldsymbol{C}_N}
\newcommand{\boldC}{\boldsymbol{C}}
\newcommand{\boldalpha}{\boldsymbol{\alpha}}

\allowdisplaybreaks

\begin{document}
	
	\title[Minimal Gap for Higher Dimensional Sequences]{Minimal Gap for Higher Dimensional Sequences}
	
	\author{Tanmoy Bera}
	\address{The Institute of Mathematical Sciences, A CI of 
		Homi Bhabha National Institute, CIT Campus, Taramani, Chennai
		600113, India.}
	\email{tanmoyb@imsc.res.in}

	\maketitle
	\begin{abstract}
		In this note, we extend the notion of minimal gaps to the higher dimensional sequences. We bound the minimal gap for $(\{\bolda_n\boldalpha\}),$ $(\{a_n\boldalpha\})$ and $(\{\bolda_n\cdot\boldalpha\})$ in terms of cardinality of the difference set of $a_n$ and $\bolda_n,$ where $a_n$ and $a_n^{(1)},\dots,a_n^{(d)}$ are sequences of distinct integers. 
	\end{abstract}
	
	\section{Introduction}
	Let $(x_n)$ be a sequence in $[0,1)$ and $N$ be a natural number. We define the minimal gap for $(x_n)$ as
	\[\delta_{\text{min}}((x_n),N):=\min\{\|x_n-x_m\|:1\leq n\neq m\leq N\},\]
	where $\displaystyle\|x\|=\min_{k\in\Z}|x+k|$ denotes the nearest integer distance of $x.$
	Let $x_1,\dots,x_N$ be $N$ independent and uniformly distributed points in the unit interval, then the size of the minimal gap is of asymptotic order $1/N^2$ almost surely as $N\to \infty$.
	
	Let $(a_n)$ be a sequence of distinct integers. In~\cite{Rudnick2018minimalgap}, Rudnick investigated the minimal gap statistic for $(\{a_n\alpha\}),$ where $\{\cdot\}$ denotes the fractional part. The additive energy of $A_N=\{a_1,\dots,a_N\}$ is defined by
	\[E(A_N)=\#\{(n_1,n_2,n_3,n_4)\in[1,N]^4\cap\N^4: a_{n_1}+a_{n_2}=a_{n_3}+a_{n_4}\}.\] 
	He showed that the minimal gap $\delta_{\text{min}}(\alpha,N)$ for $\{a_n\alpha\}$ is of similar asymptotic order as that of random sequences if the additive energy of $(a_n)$ is small. In particular, he proved that 
	for all $\eta>0,$ for almost all $\alpha,$ $$\delta_{\text{min}}(\alpha,N)>\frac{1}{N^{2+\eta}}\text{ for sufficiently large (s.l.) } N.$$
	And, if
	$E(A_N)\ll N^{2+o(1)},$ then for all $\eta>0$
	$$\delta_{\text{min}}(\alpha,N)<\frac{1}{N^{2-\eta}}\text{ for s.l. }N.$$
	
	Recently, Aistleitner, El-Baz and Munsch~\cite{Aistleitner2021minimalgap} improved these bounds. Instead of considering differences $(a_n-a_m)$ for all pairs $(m,n)$ which leads to the additive energy (in Rudnick's method), they considered only distinct differences, which improved the bounds. Suppose $(A_N-A_N)^+=\{c>0: c=a_n-a_m \text{ for some }a_n,a_m\in A_N\}.$ It is obvious that 
	\[\delta_{\text{min}}(\alpha,N)=\min\{\|z\alpha\|: z\in(A_N-A_N)^+\}.\]
	They proved the following results. Let $\epsilon>0.$ Then for almost all $\alpha\in[0,1]$
	\begin{align*}
		\delta_{\text{min}}(\alpha,N)\leq&\frac{\log_2 a_N}{C_N\log N\log_2 N}\text{ for infinitely many (i.m.) }N,
	\end{align*}
	
	and
	\[\delta_{\text{min}}(\alpha,N)\geq\frac{1}{C_N\log N(\log_2 N)^{1+\epsilon}}\text{ for s.l. }N,\]
	where $C_N=\#(A_N-A_N)^{+}.$
	(Also, the upper bound holds without the sequence-dependent term $\frac{\log_2 a_N}{\log N \log_2 N}.$) The upper bound result for s.l. $N$ is the following
	\[\delta_{\text{min}}(\alpha,N)\leq\frac{N^\epsilon}{C_N}\text{ for s.l. }N.\] 
	These bounds can be compared with the previous bounds using 
	\[N^4/C_N\leq E(A_N)\leq N^3.\]
	In recent work, Reghvim~\cite{shvoreghvim2021} studied minimal gaps for various sequences of the form $\{a_n\alpha\},$  where $a_n$ is a sequence of positive reals with some growth conditions. 
	
	In this note, we extend the notion of minimal gap to higher dimensional sequences and study upper and lower bounds for some sequences. Let $d\geq 2$ and $\|\cdot\|$ be a norm in $\R^d.$ For $\boldx\in\R^d$ we define the nearest integer distance as  $$\|\boldx\|^{(\text{nid})}=\displaystyle\min_{\boldsymbol{k}\in\Z^d}\|\boldx+\boldsymbol{k}\|.$$ 
	\begin{definition}
		Let $(\boldxn)\subseteq[0,1)^d$ be a sequence. Minimal gap for $(\boldxn)$ with respect to the norm $\|\cdot\|$ is defined by
		$$\delta_{\text{min}}^{\|\cdot\|}((\boldxn,N)=\min\{\|\boldxn-\boldx_m\|^{(\text{nid})}: 1\leq m\neq n\leq N\}$$
	\end{definition}
	
	Since all norms are equivalent in $\R^d$, it is enough to study the bounds of the minimal gap with respect to one norm only. Here we restrict ourselves the $\infty-$norm and denote the minimal gap by $\delta_{\text{min}}^\infty.$
	
	Let $(a_n^{(1)}),\ldots,(a_n^{(d)})$ be strictly increasing sequences of natural numbers and the vector $\bolda_n=(a_n^{(1)},\dots,a_n^{(d)})$. 
	The purpose of this note is to study the bounds of minimal gaps for higher dimensional sequences $(\{\bolda_n\boldalpha\})=(\{a_n^{(1)}\alpha_1\},\dots,\{a_n^{(d)}\alpha_d\})$ and $(\{a_n\boldalpha\})=\left(\left\{a_n\alpha_1\right\},\dots,\left\{a_n\alpha_d\right\}\right),$ and the linear form $(\{\bolda_n\cdot\boldalpha\})=(\{a_n^{(1)}\alpha_1+\cdots+a_n^{(d)}\alpha_d\}),$ which is a one-dimensional sequence. For the first two sequences we denote the minimal gaps by $\delta_{\text{min}}^{\infty}(\boldalpha,N)$ and for linear forms by $\delta_{\text{min}}(\boldalpha,N).$

	We define
	\begin{align*}
		&\boldA_N=\{(a_n^{(1)},\dots,a_n^{(d)}): n\leq N\},\text{ and } \\
		&(\boldA_N-\boldA_N)^+=\{\boldc\in \boldA_N-\boldA_N: \text{ all } c_i>0\}.
	\end{align*}
	Let $A=\bigcup(A_N-A_N)^+$ and $\boldA=\bigcup(\boldA_N-\boldA_N)^+.$ We write the elements of $A$ and $\boldA$ as sequences $\{z_n\}$ and $\{\boldzn\}$ of distinct natural numbers and vectors respectively satisfying the following conditions
	\begin{align}
		&\{z_1,\dots,z_{C_N}\}=(A_N-A_N)^+,\label{one dim zn}\\
		&\{\boldz_1,\dots,\boldz_{\boldCN}\} =(\boldA_N-\boldA_N)^+,\label{high dim zn}
	\end{align}
	where $C_N=\#(A_N-A_N)^+$ and $\boldCN=\#(\boldA_N-\boldA_N)^+.$
	Note that both of them satisfy $N\leq C_N,\boldCN\leq N^2.$
	Now we state our results.
	\begin{theorem}\label{theorem1}
		Let $d\geq 2$ and $\epsilon>0$. Then for almost all $\boldalpha\in[0,1]^d$ we have,
		\begin{align*} 
			\delta_{\text{min}}^\infty(\boldalpha,N)\geq \frac{1}{\boldCN^{1/d}(\log N)^{1/d}(\log_2 N)^{1/d+\epsilon}}\: \text{  for s.l. }N,
		\end{align*}
		and the upper bound of the minimal gap for $\{a_n\boldalpha\},$ also for almost all $\boldalpha\in[0,1]^d,$
		\begin{align*} 
			\delta_{\text{min}}^\infty(\boldalpha,N)\leq \frac{1}{C_N^{1/d}(\log N)^{1/d}(\log_2 N)^{1/d}}\: \text{ for i.m. }N.
		\end{align*}
	\end{theorem}
	In the above results, we obtain upper bounds for infinitely many $N.$ Now we state the results of upper bounds for sufficiently large $N.$
	Before starting the next result, for $D'\subset D:=\{1,\dots, d\},$ we define the joint additive energy of $\boldA_N^{D'}=\{\bolda_n^{D'}=(a_n^{(i)})_{i\in D'}: n\leq N\}$ by
	\[E(\boldA_N^{D'}):=\#\{\boldn\in[1,N]^4\cap\N^4:\bolda_{n_1}^{D'}+\bolda_{n_2}^{D'}=\bolda_{n_3}^{D'}+\bolda_{n_4}^{D'}\},\]
	and for $\boldv\in(\Z\setminus 0)^{D'},$ the representation function 
	\begin{align}\label{rep fun}
		\mathcal{R}_N^{D'}(\boldv)=\#\{1\leq n\neq m\leq N: \bolda_n^{D'}-\bolda_m^{D'}=\boldv\}.
	\end{align}
	Here, by $\boldv\in(\Z\setminus 0)^{D'}$ we mean $\boldv=(v_i)_{i\in D'}$ with $v_i$ non-zero. Note that $L^2$ norm of $\mathcal{R}_N^{D'}$ satisfies the bound, $\|\mathcal{R}_N^{D'}\|_2^2\leq E(\boldA_N^{D'})\leq N^3.$ 
	\begin{theorem}\label{theorem2}
		Let $d\geq 2$ and $\delta>0$. Then,
		for almost all $\boldalpha\in[0,1]^d$ the minimal gap for $(\{a_n\boldalpha\})$ satisfies
		\begin{align*} 
			\delta_{\text{min}}^\infty(\boldalpha,N)\leq \frac{(\log N)^{1+1/d+\delta}}{C_N^{1/d}}\: \text{ for s.l. }N.
		\end{align*}
		If for all $D'\subset D$ with $|D'|>d/2$ the joint additive energy satisfies $E(\boldA_N^{D'})\ll N^{2+o(1)}$ then, for almost all $\boldalpha\in[0,1]^d$ the minimal gap for $(\{\bolda_n\boldalpha\})$ satisfies
		\begin{align*} 
			\delta_{\text{min}}^\infty(\boldalpha,N)\leq \frac{N^\delta}{N^{2/d}}, \text{  for s.l. }N.
		\end{align*}
	\end{theorem}
	Let $a\neq 0, b,c$ be real numbers. Then the minimal gap for $(\{a\alpha n^2+b\alpha n+c\alpha\})$ is $<\frac{1}{N^{2-o(1)}},$ for almost all $\alpha$~\cite[Theorem 1.8]{shvoreghvim2021}.
	We consider the sequence $(\{n^2\alpha_2+n\alpha_1\}),$ and ask same questions about minimal gap for almost all $(\alpha_1,\alpha_2).$
	In the next result, we answer this question for a more general sequence of linear forms $(\{a_n^{(1)}\alpha_1+\cdots+a_n^{(d)}\alpha_d\}).$
	\begin{theorem}\label{thm3}
		Let $\epsilon>0.$ Then for almost all $\boldalpha\in[0,1]^d$ 
		\begin{align*} 
			\delta_{\text{min}}(\boldalpha,N)\geq \frac{1}{\boldCN\log N(\log_2 N)^{1+\epsilon}}, \text{ for s.l. }N,
		\end{align*}
		\begin{align*} 
			\delta_{\text{min}}(\boldalpha,N)\leq\frac{1}{\boldCN}\frac{\log_2\left(a_N^{(1)}\cdots a_N^{(d)}\right)^{1/d}}{\log N\log_2 N}, \text{ for i.m. }N.
		\end{align*}
		Moreover, following~\cite{Aistleitner2021minimalgap} we remove the sequence dependent term $\frac{\log_2\left(a_N^{(1)}\cdots a_N^{(d)}\right)^{1/d}}{\log N\log_2 N}$ from the upper bound. For almost all $\boldalpha\in[0,1]^d,$
		\begin{align*} 
			\delta_{\text{min}}(\boldalpha,N)\leq\frac{1}{\boldCN}, \text{ for i.m. }N.
		\end{align*}
	\end{theorem}
	
	\begin{theorem}\label{thm3.1}
		If $\min_{i\leq d}E(A_N^{\{i\}})\ll N^{2+o(1)}$ for $N\gg 1$ then, for all $\delta>0,$ for almost all $\boldalpha\in[0,1]^d$
		\[\delta_{\text{min}}(\boldalpha,N)\leq\frac{N^\delta}{N^2}, \text{ for s.l. }N.\]
	\end{theorem}
	
	\begin{remark*}
		Let $d=2$ and $a_n^{(i)}=n^i$ for $i=1,2.$ Then $\min_{i\leq 2}E(A_N^{\{i\}})\ll N^{2+\epsilon}$ for any $\epsilon.$ It can be easily calculated that $\boldCN=(N-1)(N-2)/2\approx N^2.$ Then above theorem implies that for almost all $\boldalpha$ the minimal gap for $n^2\alpha_2+n\alpha_1$ is 
		\begin{align*}
			&\ll \frac{1}{N^2\log N(\log_2 N)^{1/2}}\text{  for i.m. }N \text{ and,}\\ 
			&\ll \frac{N^\epsilon}{N^2}\text{  for s.l. }N. 
		\end{align*} 
	\end{remark*}
	
	Now we state our last result for the van der Corput sequence (see Section~\ref{section van der Corput} for definition).
	\begin{theorem}\label{thm4}
		For $b\geq 3,$ the minimal gap of the van der Corput sequence $g_b(n)$ is given by
		\[\frac{1}{bN}\leq\delta_{\text{min}}((g_b(n)),N)\leq \frac{b}{N},\text{ for all } N.\]
		And for $b=2,$ we have the same lower bound but upper bound becomes 
		\[\delta_{\text{min}}((g_b(n)),N)\leq \frac{1}{N},\text{ for all } N.\]
	\end{theorem}
	
	\begin{remark}
		Although we do not have any proof but we believe that for the Halton sequence $g_{\boldb}(n)$ (see Section~\ref{section van der Corput} for definition)
		\[\delta_{\text{min}}^\infty(g_{\boldb}(n),N)\asymp \frac{1}{N^{1/d}}, \text{ for all } N.\]
	\end{remark}
	
	\begin{remark}
		Instead of considering $\bolda_n$ to be an integer sequence in the higher dimensional sequence $(\{\bolda_n\boldalpha\})$ one can consider positive real sequence also, but in that case following similar arguments as in ~\cite{shvoreghvim2021} we can obtain bounds of the minimal gaps for sequences of the form  $(\{\boldsymbol{a_n}\boldalpha\}),$ where each component of $\bolda_n$ has some growth conditions.
	\end{remark}

	\section{Minimal gap for higher dimensional `integer' sequences}
	In this section, we prove Theorem~\ref{theorem1} and Theorem~\ref{theorem2}. The proofs of them follow the ideas of~\cite{Rudnick2018minimalgap} and~\cite{Aistleitner2021minimalgap}.
	
	For $a\geq 0,$ denote by $U(a)$ the set of $\boldsymbol{y}\in\R^r$ for which $0\leq y_i<a\:(i=1,\dots,r).$ For $\boldsymbol{l},\boldsymbol{l'}\in\Z^r$ and a positive integer $n,$ denote $(\boldsymbol{l},n)$ and $(\boldsymbol{l},\boldsymbol{l'})$ the gcd of $n$ and the components of $\boldsymbol{l},$ and the gcd of the components of  $\boldsymbol{l}$ and $\boldsymbol{l'}$ respectively. And $|\cdot|_\infty$ denotes the usual sup norm. In the subsequent sections, $\lambda$ denotes one, as well as a higher dimensional Lebesgue measure.
	\begin{lemma}{\cite[Theorem 1]{Gllaghar}}\label{gallaghar lemma}
		Let $d\geq 2.$ For each sequence of numbers $\psi(n)$ between $0$ and $1,$ there are infinitely many solutions $n,\boldsymbol{l}$ of 	
		\[n\boldx-\boldsymbol{l}\in U(\psi(n)),\:\:\: (\boldsymbol{l},n)=1\]
		for almost all $\boldx$ or almost no $\boldx$ according as $\Sigma \psi(n)^d$ diverges or converges.
	\end{lemma}
	\begin{proof}[Proof of Theorem~\ref{theorem1}]
		From definition of minimal gap,~\eqref{one dim zn} and~\eqref{high dim zn} we get
		\begin{align*}
			\delta_{\text{min}}^{\infty}(\boldalpha,N)=\min_{1\leq n\leq\boldCN}\|\boldzn\boldalpha\|_{\infty}^{(\text{nid})}
		\end{align*}
		\textbf{Lower bound.} Let $n\geq 1$ be an integer and 
		\begin{align*}
			\psi(n)=\frac{1}{n^{1/d}(\log\sqrt{n})^{1/d}(\log_2\sqrt{n})^{1/d+\epsilon/d}}.
		\end{align*}
		We define
		\begin{align*}
			&I_n^{(i)}=[0,1]\cap\left(\bigcup_{0\leq a\leq z_{n}^{(i)}}\left(\frac{a}{z_n^{(i)}}-\frac{\psi(n)}{z_n^{(i)}},\frac{a}{z_n^{(i)}}+\frac{\psi(n)}{z_n^{(i)}}\right)\right),
		\end{align*}
		$\displaystyle S_n=\prod_{i=1}^{d}I_n^{(i)},$
		and $S=\liminf_{n\to\infty}S_n^c.$ It is clear that $\mu\left(I_N^{(i)}\right)\leq\min(2\psi(n),1),$ which gives us $\lambda(S_n)\leq\min((2\psi(n))^d,1).$ So we obtain
		\begin{align*}
			\sum_{n=1}^{\infty}\lambda(S_n)\leq\sum_{n=1}^{\infty}\frac{1}{n\log \sqrt{n}(\log_2 \sqrt{n})^{1+\epsilon}}<\infty.
		\end{align*}
		Then, by the first Borel-Cantelli lemma $\lambda(\limsup_{n\to\infty}S_n)=0,$ so $\lambda(S)=1.$ 
		Let $\boldalpha\in S.$ Then $\|\boldzn\boldalpha\|_{\infty}\geq\psi(n)\geq\psi(n)$ for s.l. $n.$
		Note that $\psi$ is a decreasing function. So we get for sufficiently large $N$ 
		\begin{align*}
			&\|\boldzn\boldalpha\|_{\infty}\geq \psi(n)\geq\psi(\boldCN),\forall n\leq \boldCN\\
			\implies&\|(\bolda_n-\bolda_m)\boldalpha\|_{\infty}\geq \frac{1}{\boldCN^{1/d}(\log N)^{1/d}(\log_2 N)^{1/d+\epsilon}}, \text{ for all } \bolda_n,\bolda_m\in(\boldA_N-\boldA_N)^{+},
		\end{align*}
		because of~\eqref{high dim zn} and $\boldCN\leq N^2$. This proves the lower bound of $\delta_{\text{min}}^{\infty}(\boldalpha, N).$\\

		\textbf{Upper bound.}
		We define 
		\begin{align*}
			\tilde{S}_n=\left\{\boldalpha\in[0,1]^d:|z_n\boldalpha-\boldq|_{\infty}<\psi(n)\text{ for some }\boldq\in\Z^d \text{ with }(z_n,\boldq)=1\right\}
		\end{align*}
		Then by Lemma~\ref{gallaghar lemma} we get $\lambda(\limsup_{n\to\infty}\tilde{S}_n)=1$ if $\sum_{n\geq 1}\psi(n)^d=\infty.$\\
		We choose $\psi(n)=\frac{1}{(n\log n\log_2 n)^{1/d}},$ which satisfies the above divergence criterion. So,
		for any $\boldalpha\in\limsup_{n\to\infty} \tilde{S}_n$ we have 
		\begin{align*}
			\|z_n\boldalpha\|_{\infty}^{(\text{nid})}\leq \psi(n) \text{ for i.m. }n.	
		\end{align*}
		Therefore, for almost all $\boldalpha\in[0,1]^d,$ $\delta_{\text{min}}^\infty(\boldalpha,N)\leq\frac{1}{(C_N\log N\log_2 N)^{1/d}}$ for i.m. $N.$
	\end{proof}

	\subsection{Proof of Theorem~\ref{theorem2}}
	To prove upper bounds for sufficiently large $N$ we use the variance method and follow~\cite{Rudnick2018minimalgap}. 
	Let $N, M\in\N.$ We define 
	\begin{align*}
		D(N,M)(\boldalpha)=\sum_{1\leq n\leq C_N}\chi_{\|z_n\boldalpha\|_\infty\leq\frac{1}{2M}}\text{ and}\\
		\widetilde{D}(N,M)(\boldalpha)=\sum_{1\leq n\neq m \leq N}\chi_{\|(\bolda_n-\bolda_m)\boldalpha\|_\infty\leq\frac{1}{2M}}.
	\end{align*}
	It is clear that, for $(\{\bolda_n\boldalpha\})$ and $(\{a_n\boldalpha\})$ the minimal gap $\delta_{\text{min}}^{(\infty)}(\boldalpha,N)\leq 1/2M\text{ if } D(N,M)(\boldalpha)\geq 1$ and $\widetilde{D}(N,M)(\boldalpha)\geq 1$ respectively.
	
	Our aim is to show that, for almost all $\boldalpha,$  $D(N,M)(\boldalpha)\geq 1$ and $\widetilde{D}(N,M)(\boldalpha)\geq 1$ for sufficiently large $N$ and for a suitable $M.$ 
	
	Note that
	\begin{align*}
		\E[D(N,M)]=C_N/M^d,\\
		\E[\widetilde{D}(N,M)]=(N^2-N)/M^d.
	\end{align*}
	Also,
	\begin{align*}
		\chi(\boldx)\sim\sum_{\boldk\in\Z^d}c_{\boldk}e(\boldk\cdot\boldx)
	\end{align*}
	where $c_{\boldk}=c_{k_1}\cdots c_{k_d},\: c_0=1/M$  and 
	$|c_{k_i}|\leq\min(1/M,1/\pi|k_i|),i\leq d.$ Before estimating the variance we state some results. For $v,w\in\Z\setminus 0$
	\begin{align}\label{fourier coeff to gcd sum}
		\sum_{\substack{i,j\in\Z\setminus0\\ vi=wj}}|c_ic_j|\ll\frac{\log M}{M}\frac{(v,w)}{\sqrt{|vw|}}, \text{ see ~\cite[page 474]{aistleitner2017additive}}.
	\end{align}
	Let $d'\geq 1$ be an integer. For any finite $\boldA\subset\N^{d'}$ and $f: \boldA\to\R^+$ with $\|f\|_1\geq 3$ and $|\boldA|\geq\log\|f\|_1,$ the bound of the gcd sum (for $d'=1$ see~\cite[Theorem 2]{Bloom2019GCDSA} and for $d'>1$ see~\cite[Proposition 3.1]{bera das mukhopadhyay2023})
	\begin{align}\label{higher dimensional gcd sum}
		\sum_{\bolda,\boldb\in \boldA}f(\bolda)f(\boldb)\prod_{i\leq d' }\frac{(a_i,b_i)}{\sqrt{a_ib_i}}\ll\frac{|\boldA|^{\epsilon}}{(\log\|f\|_1+O(1))^{d'}}\|f\|_2^2.
	\end{align}
	
	\begin{lemma}\label{lmthm2}
		Let $\epsilon>0.$ Then	 \[\Var(D(N,M))\ll
		\frac{\log M}{M^{2d-1}}\frac{C_N^{1+\epsilon}}{\log C_N+O(1)}+\frac{(\log M)^2}{M^{2d-2}}C_N(\log_2 C_N)^2+\sum_{3\leq d'\leq d}\frac{(\log M)^{d'}}{M^{2d-d'}}C_N,\] and
		\[\Var(\widetilde{D}(N,M))\ll\sum_{\substack{D'\subseteq D\\1\leq |D'|=d'\leq d}}\frac{(\log M)^{d'}}{M^{2d-d'}}\frac{N^{\epsilon}}{(\log N+O(1))^{d'}}E(\boldA_N^{D'}).\]
	\end{lemma}
	\begin{proof}
		\begin{align*} 
			&\Var(D(N,M))=\int_{[0,1]^d}\left(\sum_{1\leq n\leq C_N}\sum_{\boldk\in\Z^d\setminus0}c_{\boldk}e((z_n\boldalpha)\cdot\boldk)\right)^2d\boldalpha\nonumber\\
			&=\sum_{\substack{D'\subseteq D\\|D'|=d'\leq d}}\sum_{\substack{1\leq n,m\leq C_N}}\sum_{\boldk,\boldk'\in(\Z\setminus0)^{D'}}c_{\boldk}c_{\boldk'}\left(\frac{1}{M}\right)^{2(d-d')}\int_{[0,1]^{d'}}e((z_n\boldalpha)\cdot\boldk-(z_m\boldalpha)\cdot\boldk')d\boldalpha\nonumber\\
			&=\sum_{\substack{D'\subseteq D\\|D'|=d'\leq d}}\left(\frac{1}{M}\right)^{2(d-d')}\sum_{1\leq n,m\leq C_N}\sum_{\substack{\boldk,\boldk'\in(\Z\setminus0)^{D'}\\z_n\boldk=z_m\boldk'}}c_{\boldk}c_{\boldk'}\nonumber\\
			&\ll\sum_{\substack{D'\subseteq D\\|D'|=d'\leq d}}\frac{(\log M)^{d'}}{M^{2d-d'}}\sum_{1\leq n,m\leq C_N}\frac{(z_m,z_n)^{d'}}{(z_nz_m)^{d'/2}}\nonumber\\
			&\ll\sum_{d'\leq d}\frac{(\log M)^{d'}}{M^{2d-d'}}\sum_{1\leq n,m\leq C_N}\frac{(z_n,z_m)^{d'}}{(z_nz_m)^{d'/2}}\nonumber\\
			&\ll \frac{\log M}{M^{2d-1}}\frac{C_N^{1+\epsilon}}{\log C_N+O(1)}+\frac{(\log M)^2}{M^{2d-2}}C_N(\log_2 C_N)^2+\sum_{3\leq d'\leq d}\frac{(\log M)^{d'}}{M^{2d-d'}}C_N,
		\end{align*} 
		here in the fourth line, we used~\eqref{fourier coeff to gcd sum} and in the last line~\eqref{higher dimensional gcd sum} for $d'=1,$ G\'al's~\cite{gal} result for $d'=2,$ and the fact that, the innermost gcd sum is $\ll C_N$ for $d'\geq 3.$ 
		
		A similar calculation shows that
		\begin{align*}
			\Var(\widetilde{D}(N,M))&=\sum_{\substack{D'\subseteq D\\|D'|=d'}}\left(\frac{1}{M}\right)^{2(d-d')}\sum_{\boldv,\boldsymbol{w}\in(\Z\setminus 0)^{D'}}\mathcal{R}_{N}^{D'}(\boldsymbol{w})\mathcal{R}_{N}^{D'}\sum_{\substack{\boldk,\boldk'\in(\Z\setminus0)^{D'}\\\boldv\boldk=\boldsymbol{w}\boldk'}}c_{\boldk}c_{\boldk'}\\
			&\ll\sum_{\substack{D'\subseteq D\\|D'|=d'}}\frac{(\log M)^{d'}}{M^{2d-d'}}\sum_{\boldv,\boldsymbol{w}\in(\Z\setminus 0)^{D'}}\mathcal{R}_{N}^{D'}(\boldsymbol{w})\mathcal{R}_{N}^{D'}(\boldv)\prod_{i\leq D'}\frac{(w^{(i)},v^{(i)})}{\sqrt{v^{(i)}w^{(i)}}}\\
			&\ll\sum_{\substack{D'\subseteq D\\|D'|=d'}}\frac{(\log M)^{d'}}{M^{2d-d'}}\frac{N^{\epsilon}}{(\log N+O(1))^{d'}}E(\boldA_N^{D'})
		\end{align*}
		The second line follows from~\eqref{fourier coeff to gcd sum}, and by~\eqref{higher dimensional gcd sum} we get the last line.
	\end{proof}
	
	\begin{proof}[Proof of Theorem~\ref{theorem2}]
		It is sufficient to prove that for $M=N^{2/d-\delta}/2,$ for almost all $\boldalpha,$ $\widetilde{D}(N,M)(\boldalpha)\geq 1$ for s.l. $N,$ and for $M=\frac{C_{N}^{1/d}}{2(\log N)^{1+1/d+\delta}},$ for almost all $\boldalpha,$ $D(N,M)(\boldalpha)\geq 1$ for s.l. $N.$
		
		First we prove it for $\widetilde{D}(N,M)(\boldalpha).$
		Let $N_k=k^C$ for some large integer $C$ such that $C\delta>2.$ Then for any $M_k<N_k^{2/d-\delta}$ we have 
		\begin{align*}
			\sum_{k\geq 1}\int_{[0,1]^d}\left|\frac{\widetilde{D}(N_k,M_k)(\boldalpha)}{(N_k^2-N_k)/M_k^{d}}-1\right|^2d\boldalpha&=\sum_{k\geq 1}\frac{M_k^{2d}\Var(\widetilde{D}(N_k,M_k))}{(N_k^2-N_k)^2}\\
			&\ll\sum_{k\geq 1}\sum_{\substack{D'\subseteq D\\|D'|=d'}}\frac{M_k^{d'}(\log M_k)^{d'}}{N_k^4}\frac{N_k^{\epsilon}}{(\log N_k+O(1))^{d'}}E(\boldA_{N_k}^{D'})\\
			&\ll\sum_{k\geq 1}\sum_{d'\leq d/2}\frac{1}{N_k^{1-2d'/d+d'\delta-\epsilon}}+\sum_{k\geq 1}\sum_{d/2<d'\leq d}\frac{1}{N_k^{2-2d'/d+\delta/2}}<\infty.
		\end{align*}
		In the last line, we used the hypothesis of the theorem and the trivial bound $E(\boldA_{N_k}^{D'})\leq N^3.$
		So, we get for almost all $\boldalpha$, $\widetilde{D}(N_k,M_k)(\boldalpha)\sim\frac{N_k^{2}}{M_k}.$ 
		
		Let $N$ be large such that $N_k\leq N<N_{k+1}.$ Then 
		\begin{align*}
			\frac{\widetilde{D}(N_{k},M)}{N_{k}^2/M}\frac{N_{k}^2}{N^2}\leq\frac{\widetilde{D}(N,M)}{N^2/M}\leq \frac{\widetilde{D}(N_{k+1},M)}{N_{k+1}^2/M}\frac{N_{k+1}^2}{N^2}.
		\end{align*}
		Since $M<N^{2/d-\delta}\sim N_k^{2/d-\delta}<N_{k+1}^{2/d-\delta},$ $\frac{N_{k+1}}{N_k}\to 1,$ and $\widetilde{D}(N_k,M_k)/\frac{N_k^2}{M_k}\to 1$ almost surely, we obtain $\frac{\widetilde{D}(N,M)}{N^2/M}\to 1$ almost surely.
		Now take $M=N^{2/d-\delta}/2$ then, for almost all $\boldalpha,$ $\widetilde{D}(N,M)(\boldalpha)\geq 1.$ Therefore, for almost all $\boldalpha,$ the minimal gap $\delta_{\text{min}}(\boldalpha,N)\leq\frac{N^\delta}{N^{2/d}},$ for s.l. large $N.$

		Similarly, we prove that $D(N, M)\geq 1$ almost surely.
		Let $N_k=2^k.$ Then for any $M_k<\frac{C_{N_k}^{1/d}}{(\log N_k)^{1+1/d+\delta}}$ by Lemma~\ref{lmthm2} we have 
		\begin{align*}
			\sum_{k\geq 1}\int_{[0,1]^d}\left|\frac{D(N,M)(\boldalpha)}{C_{N_k}/M_k^d}-1\right|^2&=\sum_{k\geq 1}\Var(D(N_k,M_k))\frac{M_k^{2d}}{C_{N_k}^2}\\
			&\ll\sum_{k\geq 1}\frac{M_k}{C_{N_k}^{1-\epsilon}}+\frac{M_k^2(\log N)^{2+\epsilon}}{C_{N_k}}+\sum_{3\leq d'\leq d}\frac{M_k^{d'}(\log N_k)^{d'}}{C_{N_k}}\\
			&\ll\sum_{k\geq 1}\frac{M_k^{d}(\log N_k)^{d}}{C_{N_k}}\ll\sum_{k\geq 1}\frac{1}{(\log N_k)^{1+\delta}}<\infty.
		\end{align*}
		So, for almost all $\boldalpha,$ $(D(N_k,M_k))\sim\frac{C_{N_k}}{M_k^d}.$ Now let $N$ be large such that $N_k\leq N<N_{k+1}.$ Following the above argument we show that $D(N,M)\sim \frac{C_N}{M^d}$ almost surely. Choosing $M=\frac{C_{N}^{1/d}}{2(\log N)^{1+1/d+\delta}}$ gives us the required result.
	\end{proof}
	
	\section{Minimal gap for linear forms}
	In this section, we study the minimal gap of the linear form $({a_n^{(1)}\alpha_1+\cdots+a_n^{(d)}\alpha_d}),$ where $a_n^{(i)}$ are sequences of distinct natural numbers and prove Theorem~\ref{thm3}. 
	
	\begin{proof}[Proof of Theorem~\ref{thm3}]
		
		Let $\bolda\neq0\in\Z^d.$ We define a map $T:[0,1]^d\to[0,1]$ in the following way
		\begin{align*}
			\boldalpha\mapsto\bolda\cdot\boldalpha\,(\text{mod } 1).
		\end{align*}
		Then~\cite[Lemma 8]{sprindzuk} says that for any $A\subset[0,1)$ Borel measurable set we have
		\begin{align}\label{lemma 8 of sprindzuk}
			\lambda(T^{-1}(A))=\lambda(A).
		\end{align}
		
		\textbf{Lower Bound.}
		For $n\geq 1,$ we define
		$$\psi(n)=\frac{1}{n\log \sqrt{n}(\log_2 \sqrt{n})^{1+\epsilon}}$$ and
		$$S_n=\{\boldalpha\in[0,1]^d:|\boldz_n\cdot\boldalpha-q|<\psi(n) \text{ for some } q\in\Z\}.$$
		So~\eqref{lemma 8 of sprindzuk} implies $\lambda(S_n)\leq\min(2\psi(n),1),$ as $\boldz_n\neq0.$ Therefore we get
		$$\sum_{n\geq 1}\lambda(S_n)<\infty.$$
		Using the first Borel-Cantelli lemma we conclude that for almost all $\boldalpha\in[0,1)^d$, $\|\boldz_n\cdot\boldalpha\|\geq\psi(n)$ for s.l. n.
		This gives us $\delta_{\text{min}}(\boldalpha,N)\geq \frac{1}{\boldCN\log N(\log_2 N)^{1+\epsilon}}, \text{ for s.l. }N.$\\
		
		\textbf{First upper bound.}
		To get the first upper bound we use~\cite[Theorem 1]{remirez} which is the generalization of the Duffin-Schaeffer conjecture for linear forms.
		For $n\geq 1,$ we define
		\[S'_n=\{\boldalpha\in[0,1]^d:|\boldz_n\cdot\boldalpha-q|<\psi(n) \text{ for some } q\in\Z, (q,\boldz_n)=1\}.\] 
		If
		\begin{align}\label{div cond}
			\sum_{n\geq 1}\frac{\phi(\gcd(\boldz_n))\psi(n)}{\gcd(\boldz_n)}=\infty,
		\end{align}
		then by~\cite[Theorem 1]{remirez} we conclude that, for almost all $\boldalpha$
		\[\|\boldz_n\cdot\boldalpha\|\leq \psi(n)\text{ for i.m. }n.\] 
		Setting $\psi(n)=\frac{\log_2(\gcd(\boldz_n))}{n\log n\log_2 n}$ together with the fact that $\phi(q)/q\gg(\log_2 q)^{-1}$ we get the required sum~\eqref{div cond} diverges. Note that for $n\leq\boldCN,$ $\gcd(\boldz_n)\leq\sqrt{z_n^{(1)}\cdots z_n^{(d)}}\leq\sqrt{a_N^{(1)}\cdots a_N^{(d)}}.$ 
		Then, using the definition of $\boldz_n$ we obtain
		\[\delta_{\text{min}}(\boldalpha,N)\leq\frac{\log_2\sqrt{a_N^{(1)}\cdots a_N^{(d)}}}{\boldCN\log N\log_2 N}\: \text{  for i.m. } N.\]
		This proves the first upper bound. \\
		
		\textbf{Upper bound without sequence dependent term.}
		Now, to remove the sequence dependent factor $\frac{\log_2\sqrt{a_N^{(1)}\cdots a_N^{(d)}}}{\log N\log_2 N}$ from the above upper bound we follow~\cite{Aistleitner2021minimalgap}. 
		In this case, for $k\geq1$ and $2^{k/2}<n\leq 2^k,$ we set
		\[\psi(n)=\frac{1}{2n+1},\]
		\[S^*_n=\{\boldalpha\in[0,1]^d:|\boldz_n\cdot\boldalpha-q|<\psi(n) \text{ for some } q\in\Z,\:p|(q,\boldz_n)\implies p\geq 4^k\}\]
		and for $q\in\Z$
		\[S_n(q)=\{\boldalpha\in[0,1]^d:|\boldz_n\cdot\boldalpha-q|<\psi(n)\}.\]
		
		\begin{lemma}\cite[Lemma 6]{victor-samju2009}\label{lemma}
			Let $p|\gcd(\boldz_n)$ and $\psi(n)\in(0,1/2).$ Then  $\lambda(\cup_{q\in\Z}S_n(pq))=\frac{2\psi(n)}{p}.$
		\end{lemma}
		\begin{lemma}\cite[Lemma 8]{victor-samju2009}\label{lemma2}
			Let $\psi(n)\in(0,1/2).$ Then  $\lambda(S^{'}_n)=2\psi(n)\frac{\phi(\boldz_n)}{\gcd(\boldz_n)}.$
		\end{lemma}
		It is clear that 
		\begin{align*}
			S^*_n=\bigcup_{\substack{q\in\Z\\p|(q,\boldz_n)\Rightarrow p>4^k}}S_n(q).
		\end{align*}
		Then using Lemma~\ref{lemma} and~Lemma~\ref{lemma2} we obtain
		\begin{align*}
			\lambda(S^*_n)&=\lambda(\bigcup_{\substack{q\in\Z\\(q,\boldz_n)=1}}S_n(q))+\lambda(\bigcup_{\substack{q\in\Z\\p|\gcd(\boldz_n)\Rightarrow p>4^k}}S_n(pq))\\
			&=\lambda(S^{'}_n)+\sum_{\substack{p|\gcd(\boldz_n)\\p>4^k}}\frac{2\psi(n)}{p}\\
			&=\frac{2\psi(n)}{\gcd(\boldz_n)}\#\{0\leq a\leq \gcd(\boldz_n): p|(a,\boldz_n)\Rightarrow p>4^k\}.
		\end{align*}
		Then by following the proof of~\cite[Lemma 2]{Aistleitner2021minimalgap} we get
		\begin{lemma}\label{lower bound of S*_n}
			For all $n\in(2^{k/2},2^k]$ 
			\[\lambda(S^{*}_n)\geq\frac{1}{n}\frac{1}{e^\gamma\log(4^k)}(1+o(1))\]
			as $k\to \infty,$ where $\gamma$ denotes the Euler-Mascheroni constant.
		\end{lemma}
		\begin{lemma}\label{lemma3}
			Let $m\neq n\in(2^{k/2},2^k],$ $d_n=\gcd(\boldz_n)$ and $d_m=\gcd(\boldz_m).$ \\
			If $\boldz_n\nparallel \boldz_m,$ then
			$$\lambda(S^{*}_n\cap S^{*}_m)=\lambda(S^{*}_n)\lambda(S^{*}_m).$$
			If $\boldz_n\parallel\boldz_m,$ then
			$$\lambda(S^{*}_n\cap S^{*}_m)\ll\frac{\sqrt{\psi(n)\psi(m)}}{4^k}+P(d_n,d_m)\lambda(S^{*}_n)\lambda(S^{*}_m)$$
			where $P(d_n,d_m)$ is defined as in~\cite[Lemma 3]{Aistleitner2021minimalgap}.
		\end{lemma}
		\begin{proof}
			If $\boldz_n\nparallel \boldz_m$ then by~\cite[Lemma 1]{remirez} we get $S^{*}_n$ and $S^{*}_m$ are independent. Therefore, $\lambda(S^{*}_n\cap S^{*}_m)=\lambda(S^{*}_n)\lambda(S^{*}_m).$

			If they are parallel then $\exists\: \bolda\in\Z^d\setminus 0$ such that $\boldz_n=d_n\bolda,$ $\boldz_m=d_m\bolda$ and $\gcd(\bolda)=1.$ We define
			\[U_n=\{\alpha\in[0,1): |d_n\alpha-q|<\psi(n), p|(q,d_n)\Rightarrow p>4^k\}.\] 
			Then $S^{*}_n\cap S^{*}_m\subseteq T^{-1}(U_n\cap U_m),$ which implies $\lambda(S^{*}_n\cap S^{*}_m)\leq\mu(U_n\cap U_m)$ (see~\eqref{lemma 8 of sprindzuk}). Using~\cite[Lemma 3]{Aistleitner2021minimalgap} we get
			\[\mu(U_n\cap U_m)\ll\frac{\sqrt{\psi(n)\psi(m)}}{4^k}+P(d_n,d_m)\lambda(U_n)\lambda(U_m).\]
			Note that $\lambda(S^{*}_n)=\mu(U_n),$ here we used~\eqref{lemma 8 of sprindzuk}. Hence, the parallel case is proved.
		\end{proof}
		\begin{lemma}(zero-one law)\cite[Theorem 1]{victor-samju2008}\label{zero one law}
			$\lambda(\limsup_{n\to\infty}S_n)\in\{0,1\}.$
		\end{lemma}
		
		By the Chung-Erd\H{o}s inequality we have
		\begin{align}\label{Chung-Erdos eqn}
			\lambda\left(\bigcup_{2^{k/2}<n\leq 2^k}S^{*}_n\right)\geq\frac{(\sum_{2^{k/2}<n\leq 2^k}\lambda(S^*_n))^2}{\sum_{2^{k/2}<n\leq 2^k}\lambda(S^*_n\cap S^*_m)}.
		\end{align}
		From Lemma~\ref{lower bound of S*_n} we get
		\begin{align}\label{numerator lower bound}
			\sum_{2^{k/2}<n\leq 2^k}\lambda(S^*_n)\geq\sum_{2^{k/2}<n\leq 2^k}\frac{1}{n}\frac{1}{e^\gamma\log 4^k}(1+o(1))\geq 0.14 
		\end{align}
		for s.l. $k.$ Also,~\cite[see equation (30)]{Aistleitner2021minimalgap}
		\begin{align}\label{upper bound of S*_n}
			\sum_{2^{k/2}<n\leq 2^k}\lambda(S^*_n)\leq 0.99.
		\end{align}
		Now we show that the denominator of the RHS of~\eqref{Chung-Erdos eqn} is $\ll 1.$ For that we divide the sum into three parts.
		\begin{align*}
			\sum_{2^{k/2}<n,m\leq 2^k}\lambda(S^*_n\cap S^*_m)=\sum_{2^{k/2}<n=m\leq 2^k}\lambda(S^*_n)+\sum_{\substack{2^{k/2}<n\neq m\leq 2^k\\ \boldz_n\parallel \boldz_m}}\lambda(S^*_n\cap S^*_m)+\sum_{\substack{2^{k/2}<n\neq m\leq 2^k\\ \boldz_n\nparallel \boldz_m}}\lambda(S^*_n\cap S^*_m)
		\end{align*}
		By Lemma~\ref{lemma3} for the third term we have
		\begin{align*}
			\sum_{\substack{2^{k/2}<n\neq m\leq 2^k\\ \boldz_n\nparallel \boldz_m}}\lambda(S^*_n)\lambda(S^*_m)\ll(\sum_{2^{k/2}<n\leq 2^k}\lambda(S^*_n))^2.
		\end{align*}
		Therefore, by~\eqref{upper bound of S*_n} the first and third terms are $\ll 1.$ 
		Again, following Lemma~\ref{lemma3} we get that the second term is bounded above by
		\begin{align*}
			\sum_{2^{k/2}<n,m\leq 2^k}\frac{\sqrt{\psi(n)\psi(m)}}{4^k}+\sum_{2^{k/2}<n,m\leq 2^k}P(d_n,d_m)\lambda(S^{*}_n)\lambda(S^{*}_m)
		\end{align*}
		It is easy to see that $\sum_{2^{k/2}<n,m\leq 2^k}\frac{\sqrt{\psi(n)\psi(m)}}{4^k}\ll 1.$ For the second sum, using arguments given in~\cite[see pages 3869-3871]{Aistleitner2021minimalgap} we obtain
		\[\sum_{2^{k/2}<n,m\leq 2^k}P(d_n,d_m)\lambda(S^{*}_n)\lambda(S^{*}_m)\ll 1.\]
		Therefore, gluing it all together we have
		\begin{align}\label{dinominator}
			\sum_{2^{k/2}<n,m\leq 2^k}\lambda(S^*_n\cap S^*_m)\ll 1, \text{ for s.l. }k.
		\end{align}
		Now, together with ~\eqref{Chung-Erdos eqn},~\eqref{numerator lower bound} and~\eqref{dinominator} we obtain
		\begin{align*} 
			\lambda\left(\bigcup_{2^{k/2}<n\leq 2^k}S^{*}_n\right)\gg 1, \text{ for s.l. }k.
		\end{align*}
		This implies that for large $k,$ $\lambda(\limsup_{n\to \infty} S^*_n)>0.$ Since $S^*_n\subset S_n,$ we get $\lambda(\limsup_{n\to \infty} S_n)>0.$ Then by Lemma~\ref{zero one law} we  have
		\[\lambda(\limsup_{n\to \infty} S_n)=1.\]
		Thus, for almost all $\boldalpha$ there are infinitely many $n$ such that
		\[\|\boldz_n\cdot\boldalpha\|\leq\frac{1}{2n+1}.\]
		Therefore,
		\[\delta_{\text{min}}(\boldalpha,N)=\min_{1\leq n\leq\boldCN}\|\boldz_n\cdot\boldalpha\|\leq\frac{1}{2\boldC_{N-1}+1}\leq\frac{1}{\boldCN}\text{ for i.m. } N.\]
		Here we used the fact that $\boldCN\leq \boldC_{N-1}+N.$
	\end{proof}
	
	\begin{proof}[Proof of Theorem~\ref{thm3.1}]
		Let $D(N,M)(\boldalpha)=\sum_{1\leq n\neq m\leq N}\chi_{\|(\bolda_n-\bolda_m)\cdot \boldalpha\|\leq\frac{1}{2M}}.$ 
		Our aim is to show that for almost all $\boldalpha,$ $D(N,M)(\boldalpha)\geq 1.$
		\begin{align*}
			&\int_{[0,1]^d}\left(D(N,M)-\frac{(N^2-N)}{M}\right)^2d\boldalpha\\
			&=\sum_{\substack{1\leq n\neq m\leq N\\1\leq k\neq l\leq N}}\sum_{j,t\in\Z\setminus 0}c_tc_j\int_{[0,1]^d}e(j(\bolda_n-\bolda_m)\cdot\boldalpha-t(\bolda_k-\bolda_l)\cdot\boldalpha)d\boldalpha.\\
			&\ll\sum_{\substack{1\leq n\neq m\leq N\\1\leq k\neq l\leq N}}\sum_{\substack{j,t\in\Z\setminus 0\\j(\bolda_n-\bolda_m)=t(\bolda_k-\bolda_l)}}|c_tc_j|\\
			&\ll \min_{i\leq d}\sum_{v,w\in\Z\setminus 0}\mathcal{R}_N^{\{i\}}(v)\mathcal{R}_N^{\{i\}}(w)\sum_{\substack{j,t\in\Z\setminus 0\\jv=tw}}|c_tc_j|\\
			&\ll \frac{\log M }{M}\frac{N^{\epsilon}\min_{i\leq d}E(A_N^{\{i\}})}{(\log N+O(1))},\\
		\end{align*}
		for any $\epsilon>0$. We used~\eqref{higher dimensional gcd sum} in the last line.
		Under the hypothesis similar arguments as in Theorem~\ref{theorem2} show that for almost all $\boldalpha,$ $\delta_{\text{min}}(\boldalpha,N)\leq\frac{N^\delta}{N^2},$ for s.l. large $N.$ 
	\end{proof}

	\section{Minimal gap for van der Corput sequence}\label{section van der Corput}
	Let $b\geq 2$ be an integer. For any $n\in\N$ we can write it in base $b$ as
	\[n=\sum_{i\geq 0}a_ib^i,\]
	where $0\leq a_i<b.$ The van der Corput sequence $(g_b(n))$ is defined by
	\[g_b(n)=\sum_{i\geq 0}\frac{a_i}{b^{i+1}}.\]
	Let $b_1,\dots,b_d\geq 2$ be pairwise coprime natural numbers. The Halton sequence $g_{\boldb}(n)$ is the $d$-dimensional sequence $(g_{b_1}(n),\dots,g_{b_d}(n)).$
	\begin{proof}[Proof of Theorem~\ref{thm4}]
		Let $k\geq1$ and $N=b^k-1.$ Suppose $1\leq m\neq n\leq N$ and
		\begin{align*}
			&n=\sum_{i=0}^{k-1}a_ib^i,\:m=\sum_{i=0}^{k-1}a'_ib^i.
		\end{align*}
		Then
		\begin{align*}
			&b^k(g_b(n)-g_b(m))=\sum_{i=0}^{k-1}(a_i-a'_i)b^{k-(i+1)}=A-C\\
			\implies&\|g_b(n)-g_b(m)\|=\left\|\frac{A-C}{b^k}\right\|=\min(|A-C|/b^k,1-|A-C|/b^k).
		\end{align*}
		Since $|A-C|\geq 1$ and $|A-C|\leq b^k-1$ we have
		\[\|g_b(n)-g_b(m)\|\geq\frac{1}{b^k},\]
		which implies that 
		\begin{align}\label{subseq lb}
			\delta_{\text{min}}((g_b(n)),b^k-1)\geq \frac{1}{b^k}.
		\end{align}
		Now let $N=b^k,$ $n=b^{k-1}$ and $m=2b^{k-1}.$ Then 
		\begin{align*}
			&\|g_b(n)-g_b(m)\|=\frac{1}{b^k} \text{ for }b\geq 3,
			\text{ and for } b=2\,\|g_2(n)-g_2(m)\|=\frac{1}{2^{k+1}},
		\end{align*}
		which imply that 
		\begin{align}\label{subseq ub}
			\delta_{\text{min}}((g_b(n)),b^k)\leq \frac{1}{b^k}\text{ for }b\geq 3,\text{ and }
			\delta_{\text{min}}((g_2(n)),2^k)\leq \frac{1}{2^{k+1}}.
		\end{align}
		Now we want to bound the minimal gaps for general $N.$
		Let $N\geq 1$ such that $b^{K-1}\leq N<b^K$ for some $K\geq 1.$ Then 
		\begin{align}\label{minimal gap inequality}
			\delta_{\text{min}}((g_b(n)),b^K-1)\leq\delta_{\text{min}}((g_b(n)),N)\leq\delta_{\text{min}}((g_b(n)),b^{K-1})
		\end{align}
		So, \eqref{subseq lb}, \eqref{subseq ub} and \eqref{minimal gap inequality} together give us
		\begin{align*}
			\frac{1}{bN}\leq\frac{1}{b^K}\leq\delta_{\text{min}}((g_b(n)),N)\leq\frac{1}{b^{K-1}}\leq\frac{b}{N}\text{ for $b\geq 3,$ and }\delta_{\text{min}}((g_2(n)),N)=\frac{1}{2^{K}}.
		\end{align*}
		This proves the theorem.
	\end{proof}
	\section{Acknowledgement}
	The author would like to thank his thesis supervisor Prof. Anirban Mukhopadhyay for fruitful discussions and many suggestions on an earlier version of this note.  Also, he would like to thank Prof. Christoph Aistleitner for several suggestions.

	\bibliographystyle{abbrv}
\end{document}